\documentclass[oneside, 12pt]{amsart} 
\usepackage{multirow}

\usepackage[utf8]{inputenc}
\usepackage[margin=1in]{geometry}
\usepackage{amsthm}
\usepackage{enumitem}
\usepackage[OT2,T1]{fontenc}
\usepackage{comment}
\usepackage{amscd,amssymb, enumitem, amsmath,mathrsfs, amsfonts, amsaddr}
\usepackage{url}
\usepackage{tikz}
\usepackage{fullpage}
\usepackage{microtype}
\usetikzlibrary{decorations.pathreplacing}
\usepackage[backref=page]{hyperref}
\usepackage{multirow}

\DeclareSymbolFont{cyrletters}{OT2}{wncyr}{m}{n}
\DeclareMathSymbol{\Sha}{\mathalpha}{cyrletters}{"58}

\usepackage[utf8]{inputenc}

\newtheorem{theorem}{Theorem}[section]
\newtheorem{lemma}[theorem]{Lemma}
\newtheorem{proposition}[theorem]{Proposition}
\newtheorem*{proposition*}{Proposition}

\newtheorem*{questiona*}{Question A}
\newtheorem*{questionb*}{Question B}
\newtheorem*{theorem*}{Theorem}

\theoremstyle{definition}

\newtheorem{conjecture}[theorem]{Conjecture}

\newtheorem{remark}[theorem]{Remark}

\theoremstyle{remark}

\title{On a conjecture of Deines}

\author{Mentzelos Melistas}
\address{University of Twente, Department of Applied Mathematics, Drienerlolaan 5, 7522 NB Enschede, The Netherlands}

\begin{document}

\maketitle

\begin{abstract}
    Two elliptic curves defined over $\mathbb{Q}$ are called discriminant twins if they have the same minimal discriminant and the same conductor. Deines, in 2014, conjectured that there exist infinitely many semi-stable non-isogenous discriminant twins. In this article we present an explicit infinite family of semi-stable non-isogenous discriminant twins, providing a proof for Deines’ conjecture.
\end{abstract}

\section{introduction}

We call two elliptic curves $E_1/\mathbb{Q}$ and $E_2/\mathbb{Q}$ discriminant twins if they have the same minimal discriminant and the same conductor. The study of these curves originates in Deines’ PhD thesis \cite{Deines2014} and is motivated by earlier work of Ribet and Takahashi, from \cite{ribettakahashi}. More precisely, discriminant twins are related to the problem of determining the optimal
quotient in a semi-stable elliptic curve isogeny class parametrized by a Shimura variety. 

If $E_1/\mathbb{Q}$ and $E_2 /\mathbb{Q}$ are discriminant twins, then by considering quadratic twists at primes of good reduction of $E_1/\mathbb{Q}$ (and hence of $E_2 /\mathbb{Q}$) we can generate infinitely many more discriminant twins. Such quadratic twists will have primes of additive reduction. Therefore, it makes sense to consider discriminant twins which are in addition semi-stable.

In \cite{Deines2018} Deines classified all semi-stable discriminant twins $E_1/\mathbb{Q}$ and $E_2/\mathbb{Q}$ under the additional condition that $E_1/\mathbb{Q}$ and $E_2/\mathbb{Q}$ are isogenous. In particular, Deines proved that there exist only finitely many semi-stable isogenous discriminant twins. The next step is to study semi-stable discriminant twins which are not isogenous. After performing computations on all curves up to conductor $299998$ in the Cremona database \cite{cremonadata}, Deines was led to the following conjecture.

\begin{conjecture}(see \cite[Conjecture 6.4.1]{Deines2014} and \cite[Conjecture 15]{Deines2018})
    There are infinitely many semi-stable non-isogenous discriminant twins.
\end{conjecture}

The main result in this article is a proof of the above conjecture. In fact we are able to achieve a more precise result; we produce an explicit infinite family of discriminant twins. Before we precisely state our main theorem we need some preparation. Let $t \geq 1$ be an integer and consider the pairs of numbers 
    \begin{align*}
        &A=(364t+1)(16(364t+1)-(364t+5)) \\
        &B=16(16(364t+5)^2-(364t+1)^2)
    \end{align*}
    and 
    \begin{align*}
        &A'=16(364t+5)^2-(364t+1)^2\\
        &B'=16(364t+5)(16(364t+1)-(364t+5)).
    \end{align*} Consider also the elliptic curves $E_{(A,B)}/\mathbb{Q}$ and $E_{(A',B')}/\mathbb{Q}$ given by the Weierstrass equations $$E_{(A,B)} \: : \: y^2=x(x-A)(x+B)$$ and $$E_{(A',B')} \: : \: y^2=x(x-A')(x+B').$$

    Our main theorem is the following.

    \begin{theorem}\label{maintheorem} For every integer $t \geq 1$ the curves $E_{(A,B)}/\mathbb{Q}$ and $E_{(A',B')}/\mathbb{Q}$ are semi-stable non-isogenous discriminant twins. \\
    In particular, there exist infinitely many semi-stable discriminant twins.
    \end{theorem}

Two remarks are in order.

\begin{remark}
    The minimal discriminants that we get for different (even small) values of $t$ are rather large. For example, the smallest minimal discriminant (coming from $t=1$) is equal to $$3^4 \cdot 5^2 \cdot  7^2 \cdot  11^2 \cdot  41^2 \cdot  73^2 \cdot  101^2 \cdot  263^2 \cdot  5471^2 \cdot  94099^2.$$ This leads us to believe that there must exist more infinite families of semi-stable non-isogenous discriminant twins.
\end{remark}

\begin{remark}
     Discriminant twins over general number fields have been studied in \cite{bbhdhr} and \cite{dhinrw}. By considering base extensions of our above family we can get infinitely many semi-stable discriminant twins defined over more general number fields.
\end{remark}

\section{Proof of the main theorem}

Before we proceed to the proof of Theorem \ref{maintheorem} we need two lemmas; one on reduction properties of elliptic curves that have a Weierstrass equation of a specific form and one guaranteeing the existence of a certain family of integers satisfying some specific Diophantine conditions.

\begin{lemma}\label{lemmareduction}
    Let $A,B$ be coprime integers with $A \equiv -1 ( \text{mod}\:  4)$ and ord$_2(B)=4$. Consider the elliptic curve $E/\mathbb{Q}$ given by $$y^2=x(x-A)(x+B).$$ Then the minimal discriminant of $E/\mathbb{Q}$ is equal to $$\Delta_E= \frac{1}{2^8}A^2B^2(A+B)^2.$$ Moreover, we have that
    \begin{enumerate}
        \item $E/\mathbb{Q}$ has good reduction modulo $2$.
        \item If $p$ is an odd prime with $p \nmid AB(A+B)$, then $E/\mathbb{Q}$ has good reduction modulo $p$.
        \item If $p$ is an odd prime with  $p \mid AB(A+B)$, then $E$ has multiplicative reduction modulo $p$.
    \end{enumerate}
\end{lemma}
\begin{proof}
   The discriminant and the $c_4$-invariant of the given Weierstrass equation are $$\Delta=2^4 B^2 A^2 (A + B)^2 \quad \text{and} \quad c_4=2^4(A^2 + AB + B^2),$$ respectively. If $p$ is an odd prime such that $p \nmid AB(A+B)$, then $p \nmid \Delta$ and, hence, $E/\mathbb{Q}$ has good reduction modulo $p$. Moreover, since $A,B$ are coprime, we see that if an odd prime $p$ divides $AB(A+B)$, then $p \nmid c_4$ and, hence $E/\mathbb{Q}$ has multiplicative reduction modulo $p$. This proves Parts $(i)$ and $(ii)$. On the other hand, Part $(i)$ is exactly \cite[Lemma 2]{diamondkramer}. 
   
   Finally, the fact that $\Delta_E$ is the minimal discriminant follows from the previous paragraph.
\end{proof}

\begin{lemma}\label{lemmaconditions}
    Let $t \geq 1$ be an integer and consider the pairs of numbers 
    \begin{align*}
        &A=(364t+1)(16(364t+1)-(364t+5)) \\
        &B=16(16(364t+5)^2-(364t+1)^2)
    \end{align*}
    and 
    \begin{align*}
        &A'=16(364t+5)^2-(364t+1)^2\\
        &B'=16(364t+5)(16(364t+1)-(364t+5)).
    \end{align*}
    Then the following conditions are true for $A,B$ and $A',B'$.
    \begin{enumerate}
        \item $AB(A+B)=A'B'(A'+B')$
        \item $A^2+AB+B^2 \neq A'^2+A'B'+B'^2$
        \item ord$_2(B)=4=$ord$_2(B')$
        \item $A \equiv -1 ( \text{mod}\:  4)$ and $A' \equiv -1 ( \text{mod}\:  4)$
        \item $A,B$ are coprime
        \item $A',B'$ are coprime
        \item As $t \rightarrow \infty$ we have that $AB(A+B) \rightarrow \infty$.
    \end{enumerate}
\end{lemma}

\begin{proof}
   The proofs of $(iii)$ and $(iv)$ are obvious from the definition of $A,B$ and $A',B'$.

    Proof of $(i)$: A very long computation by hand (or a short computation in SAGE \cite{sagemath}) shows that $AB(A+B)$ and $A'B'(A'+B')$ are both equal to 
    \begin{align*}
        2135261074938421248000t^6 + 130710551904763084800t^5 + 3118842418505748480t^4 \\  + 36219784702087168t^3 + 207018434558208t^2 + 516389304704t + 449082480
    \end{align*}

    Proof of $(ii)$: Using SAGE \cite{sagemath} we see that 
    \begin{align*}
        A^2+AB+B^2=1078327546732800t^4 + 60689835198720t^3 + 1274643054048t^2 \\ + 11819759952t + 40825801
    \end{align*} and
    \begin{align*}
        A'^2+A'B'+B'^2=1078327546732800t^4 + 34906855576320t^3 + 354967118688t^2\\ + 1177870512t + 1284721.
    \end{align*}
    Thus we have that 
    \begin{align*}
        (A^2+AB+B^2)-(A'^2+A'B'+B'^2)= 25782979622400t^3  + 919675935360t^2 \\ + 10641889440t + 39541080 >0,
    \end{align*} because $t$ is positive.

    Proof of $(v)$ and $(vi)$: Let $$u=364t+5 \quad \quad \text{and} \quad \quad v=364t+1.$$
By the definition of $A$ and $B$ we see that that $$A=v(16v-u) \quad  \text{and} \quad B=16(16u^2-v^2).$$   
Similarly by the definition of $A'$ and $B'$ we have that $$A'=16u^2-v^2  \quad \text{and} \quad B'=16v(16v-u).$$

We first show that $A$ and $B$ are coprime. Let $p$ be a prime that divides both $A$ and $B$. Since $A$ is odd, then $p$ must be odd. Moreover, since $A=v(16v-u)$, we find that either $p \mid v$ or $p \mid 16u-v$.\\
\underline{Case 1}: Assume $p \mid v$. Then since $p \mid B$ and $B=16(16u^2-v^2)$ we see that $p \mid 16u^2-v^2$. It follows that $p \mid u$, but this is impossible since $u,v$ are coprime.\\
\underline{Case 2}: Assume $p \mid 16v-u$. Then we have that $$16u^2-v^2 \equiv 16(16v)^2-v^2\equiv (16^3-1)v^2 \equiv 4095v^2=3^2\cdot 5 \cdot 7 \cdot 13 \cdot v^2 (\text{mod} \: p).$$ By the previous case have that $p \nmid v$. Therefore, $p$ must be equal to $3,5,7,$ or $13$. We will now show that $p$ cannot be equal to any of these numbers. First, note that $$16v-u=16(364t+1)-(364t+5)=15(364t+1)+364t+1-364t-5=15v-4.$$ Recall also that $p \mid 16v-u$, i.e., $$15v-4 \equiv 16v-u \equiv 0 (\text{mod} \: p).$$ On the other hand, if $p =3 $ or $5$, then $$15v-4 \equiv -4 (\text{mod} \: p),$$ which is a contradiction. Finally, if $p$ is equal to $7$ or $13$, we have that $$15v-4 \equiv 15(364t+1)-4 \equiv 15(2^2 \cdot 7 \cdot 13 \cdot t+1)-4 \equiv 15 -4 \equiv 11 (\text{mod} \: p),$$ which is also a contradiction. This proves that $A$ and $B$ are coprime.

To prove that $A'$ and $B'$ are coprime we note that the factors of $A$ and $B$ are the same as the factors of $A'$ and $B'$, but rearranged. Since $A$ and $B$ have no common factors, then the same will be true for the factors of $A'$ and $B'$.

Proof of $(vii)$: This follows from the formula for $AB(A+B)$ in terms of $t$ from the proof of Part $(ii)$.

\end{proof}

 We are now ready to proceed to the proof of our main theorem.
 
\begin{proof}[Proof of Theorem \ref{maintheorem}]
    Let $t$ be a positive integer and let $A,B,A',B'$ be as in Lemma \ref{lemmaconditions}. Consider now the elliptic curves $E_{(A,B)}/\mathbb{Q}$ and $E_{(A',B')}/\mathbb{Q}$. According to Lemma \ref{lemmareduction} both elliptic curves are semi-stable. Moreover, again from Lemma \ref{lemmareduction} we see that $E_{(A,B)}/\mathbb{Q}$ has minimal discriminant $$\frac{1}{2^8}A^2B^2(A+B)^2$$ and that the curve $E_{(A',B')}/\mathbb{Q}$ has minimal discriminant $$\frac{1}{2^8}A'^2B'^2(A'+B')^2.$$ Since by Part $(i)$ of Lemma \ref{lemmaconditions} we have that $$AB(A+B)=A'B'(A'+B'),$$ we find that the curves $E_{(A,B)}/\mathbb{Q}$ and $E_{(A',B')}/\mathbb{Q}$ have the same minimal discriminant. Since they are both semi-stable, they must also have the same conductor.

    On the other hand, the $j$-invariants of  $E_{(A,B)}/\mathbb{Q}$ and $E_{(A',B')}/\mathbb{Q}$ are equal to $$j(E_{(A,B)})=256\frac{(A^2+AB+B^2)^3}{A^2B^2(A+B)^2}$$ and $$j(E_{(A',B')})=256\frac{(A'^2+A'B'+B'^2)^3}{A'^2B'^2(A'+B')^2},$$ respectively. It follows from Part $(ii)$ of Lemma \ref{lemmaconditions} that $j(E_{(A,B)}) \neq j(E_{(A',B')})$. Thus, the curves $E_{(A,B)}/\mathbb{Q}$ and $E_{(A',B')}/\mathbb{Q}$ cannot be isomorphic. Moreover, by Part $(vii)$ we have that  their common minimal discriminant  $$\frac{1}{2^8}A^2B^2(A+B)^2$$ approaches $\infty$ as $t \rightarrow \infty$. Thus we really have an infinite family of semi-stable discriminant twins.

    It remains to show that $E_{(A,B)}/\mathbb{Q}$ and $E_{(A',B')}/\mathbb{Q}$ cannot be isogenous. It follows from \cite[Theorem 1]{Deines2018} that if $E_{(A,B)}/\mathbb{Q}$ and $E_{(A',B')}/\mathbb{Q}$ are isogenous, then their conductor is at most $37$. Since the there are only finitely many curves of conductor up to $37$, the number of possible minimal discriminants is bounded. However, by Part $(vii)$ we have that the minimal discriminant $$\frac{1}{2^8}A^2B^2(A+B)^2$$ of $E_{(A,B)}/\mathbb{Q}$ (and of $E_{(A',B')}/\mathbb{Q}$) approaches infinity as $t \rightarrow \infty$. Therefore, the curves $E_{(A,B)}/\mathbb{Q}$ and $E_{(A',B')}/\mathbb{Q}$ cannot be isogenous.

    Finally, the last sentence of the theorem is obvious. This completes our proof.

\end{proof}

In \cite[Question 6.1.2]{Deines2014} Deines asked whether infinitely many $j$-invariants occur among all discriminant twin pairs. Below we give a positive answer to that question as well.

\begin{proposition}
    The set $$\{ j(E) \:\:| \quad  E/\mathbb{Q} \:\: \text{ is a discriminant twin} \}$$ is infinite.
\end{proposition}
\begin{proof}
   Using SAGE \cite{sagemath} we can check that the $j$-invariant $j(E_{(A,B)})$ of the curve $E_{(A,B)}/\mathbb{Q}$ is a non-constant rational function in $\mathbb{Q}(t)$. Therefore, $j(E_{(A,B)})$ cannot take the same values infinitely many times and, hence, as $t$ varies in positive integers, the $j$-invariants of the curves $E_{(A,B)}/\mathbb{Q}$ form an infinite set.  This proves our proposition.
\end{proof}

\begin{remark}
    Our family is a family of semi-stable elliptic curves with full $2$-torsion. It would be interesting if other infinite families of semi-stable non-isogenous discriminant twins with other torsion structures can be constructed.
\end{remark}

\bibliographystyle{plain}
\bibliography{bibliography.bib}
\end{document}